\documentclass{amsart}

\usepackage{graphicx}
\usepackage{amsmath}
\usepackage{amssymb}
\usepackage{todonotes}
\newtheorem{theorem}{Theorem}[section]
\newtheorem{lemma}[theorem]{Lemma}

\theoremstyle{definition}
\newtheorem{definition}[theorem]{Definition}
\newtheorem{example}[theorem]{Example}

\newtheorem{corollary}[theorem]{Corollary}
\newtheorem{proposition}[theorem]{Proposition}

\theoremstyle{remark}
\newtheorem{remark}[theorem]{Remark}

\numberwithin{equation}{section}

\begin{document}

\title[Classical simple Lie $2$-algebras of toral rank $3$ ]{Classical simple Lie $2$-algebras of toral rank $3$ and a contragredient Lie 2-algebra of toral rank $4$.}

%    Information for first author
\author{Carlos R. Payares Guevara}
%    Address of record for the research reported here
\address{
Faculty of basic sciences , Universidad Tecnol\'ogica de Bol\'var, Cartagena de Indias - Colombia}
%Rouge, Louisiana 70803}
%    Current address
%\curraddr{Department of Mathematics and Statistics,
%Case Western Reserve University, Cleveland, Ohio 43403}
\email{cpayares@utb.edu.co}
%    \thanks will become a 1st page footnote.
%\thanks{The first author was supported in part by NSF Grant \#000000.}

 %   Information for second author
\author{Fabi\'an A. Arias Amaya.}
\address{Faculty of basic sciences, Universidad Tecnol\'ogica de Bol\'ivar, Cartagena de Indias - Colombia}
%Australian National University, Canberra ACT 2601, Australia}
\email{farias@utb.edu.co}
%\thanks{Support information for the second author.}

%    General info

\date{January 28, 2019.}

%\dedicatory{This paper is dedicated to our advisors.}


\begin{abstract}
In this paper we show there are no classical type simple Lie $2$-algebras with toral rank odd and  we  also show  that the simple contragredient Lie $2$-algebra $G(F_{4, a})$ of dimension $34$  has toral rank $4$, and we give  the Cartan decomposition of $G(F_{4, a})$.
%This paper is a sample prepared to illustrate the use of the American
%Mathematical Society's \LaTeX{} document class \texttt{amsart} and
%publication-specific variants of that class for AMS-\LaTeX{} version 2.
\end{abstract}

\maketitle

\section*{Introduction}

The classification  of the simple Lie algebras over an algebraically closed field of characteristic $p$  with $p \in\{2,3\}$ is still an open problem. In characteristic $2$, S. Skryabin in \cite{gr2} showed that all simple Lie algebras on an algebraically closed field of characteristic $2$ have absolute toral rank greater than or equal to $2$ (see also \cite{gr21345}). Later, A. Premet and  A. Grishkov classified Lie algebras of absolute toral rank $2$. They annouced in \cite{gr} (work in progress) the following result:
All  finite dimensional simple Lie algebra over an algebraically closed field of characteristic $2$ of absolute toral rank $2$ are classical of dimesion $3$, $8$, $14$ or $26$. In particular, all  finite dimensional simple Lie $2$-algebra over a field of characteristic $2$ of (relative) toral rank $2$ is isomorphic to $A_{2}$, $G_{2}$ or $D_{4}.$ 
When the rank absolute is greater than or equal to $3$ the problem of classification is still open. The main obstacle in this problem is the lack of examples. 

\medskip

In this paper we calculate the toral rank of the Classical simple Lie $2$-algebras of type $X_{l}\in\{A_{l},B_{l},C_{l},D_{l},\mathfrak{g}_{2},\mathfrak{f}_{4},\mathfrak{e}_{6},\mathfrak{e}_{7},\mathfrak{e}_{8}\}$, i.e., quotients of Chevalley algebras over a field of characteristic $2$, module the center.  
As a consequence, we obtain  our first main result:

\medskip

\noindent \textbf{Theorem 1.} 
There are no classical type simple Lie $2$-algebra  of odd toral rank.
In particular, there are no classical type simple Lie $2$-algebra  of odd toral rank.

\medskip
V. Kac in \cite{VK} showed that for $p > 3$   every simple finite dimensional contragredient Lie algebra is isomorphic to one of the simple Lie algebras of the classical type. If $p = 2$, this is no longer true and  the classification of simple finite dimensional contragredient Lie algebras  is still considered an open problem. In the last section we proved that the simple contragredient Lie $2$-algebra of dimension $34$ constructed by V. G. Kac and V. Ve\u{i}sfe\u{i}ler  in \cite{ka} has toral rank $4$
   and we calculate the dimension of the root spaces of this contragredient Lie algebra. More specifically, we  have:

\medskip

\noindent \textbf{Theorem 2.}
The simple contragredient Lie $2$-algebra of dimension $34$ with Cartan matrix    \[F_{4,a}:= \left( \begin{array}{cccc} 
0 & 1 & 0 & 0\\
a & 0 & 1 & 0 \\
0 & 1 & 0 & 1 \\
0 & 0 & 1 & 0 \end{array} \right),\] 
 which is denoted by $G(F_{4,a})$, has toral rank $4$. Furthermore,   the Cartan descomposition of $G(F_{4,a})$ with respect to the $4$-dimensional torus $T(\mathfrak{h})$
  is
  \[G(F_{4,a})=T(\mathfrak{h})\oplus\left(\underset{\xi\in G}\oplus\mathfrak{g}_\xi\right)\] where  $G:=\left\langle \alpha,\beta,\gamma,\lambda\right\rangle$ is an elementary abelian group of order $16$, and $\text{dim}_{K}(\mathfrak{g}_\xi)=2$, for all $\xi\in G$.

\medskip

The only  classical type simple Lie $2$-algebra of toral rank $4$ over a algebraically closed field of characteristic $2$ are  $\mathfrak{sl}_{5}(K)$, $\mathfrak{psl}_{6}(K)$, $\mathfrak{sp}_{10}(K)^{(2)}$ and\\ $\mathfrak{sp}_{12}(K)^{(2)}/\mathfrak{z}(\mathfrak{gl}_{12}(K))$ (see Corollary \ref{COR}). Theorem (2) gives us an example of a no classical simple Lie $2$-algebra, which should be taken into account in future investigations related to the problem of classifying the simple Lie 2-algebras of toral rank 4.
   
   \medskip

In  section 1 we present some basic definitions and well-known results that will be used throughout the work. In Section 2 and 3 we show that the linear special Lie algebra $\mathfrak{sl}_{n+1}(K)$ and the symplectic Lie algebra $\mathfrak{sp}_{2m}(K)$   over an algebraically closed field of characteristic $2$ are Lie $2$-algebra, and we study the simplicity of these algebra (Theorem \ref{isa 1} and \ref{isa 2}).   In section 4 we show that the orthogonal Lie algebra $\mathfrak{o}_{n}(K)^{(1)}$ is not a Lie $2$-algebra. In section 5 we list all classical type simple Lie $2$-algebras, we calculate its toral rank, and we conclude that are  there  no classical type simple Lie $2$-algebas with odd  toral rank. Finally, in last section  we show that the simple Lie $2$-algebra of dimension $34$ constructed by V. G. Kac and V. Ve\u{i}sfe\u{i}ler  in \cite{ka} has toral rank $4$, and we also give the Cartan decomposition of this algebra. 

\section{Preliminaries.}

Throughout this paper all algebras are defined over a fixed algebraically closed field $K$ of charecteristic $2$ containing the prime field  $\mathbb{K}_{2}$ and $\mathfrak{g}$ is any Lie algebra of finite dimension on $K$. We start with some basic definitions and known facts.

\subsection{Simple Lie $2$-algebra.}

\begin{definition}
A \textit{Lie $2$-algebra} is a pair $(\mathfrak{g}, [2])$ where $\mathfrak{g}$ is a Lie algebra over $K$, and $[2]: \mathfrak{g}\rightarrow  \mathfrak{g} $, $ a\mapsto a ^ {[2]} $ is a map (called $2$-map) such that:
\begin{enumerate}
 \item  $ {(a + b)} ^ {[2]} = {a} ^ {[2]} + {b} ^ {[2]} + [a, b] $,\quad $\forall\,\,  a, b\in\mathfrak{g} $.
 \item  $ (\lambda a) ^ {[2]} = \lambda ^ {2} a ^ {[2]} $,\quad $ \forall\,\, \lambda \in K $,\quad $ \forall\,\, a \in \mathfrak{g}. $
 \item $\text{ad}({b}^{[2]})=(ad(b))^{2} $,\quad $ \forall\,\, b \in \mathfrak{g} $.
 \end{enumerate}
\end{definition}
If the center, $\mathfrak{z}(\mathfrak{g})$,  of $\mathfrak{g}$  is zero  and a $2$-map $[2]:\mathfrak{g}\to \mathfrak{g}$ exists, it is unique. A Lie $2$-algebra  $(\mathfrak{g}, [2])$ is called a \textit{simple Lie $2$-algebra}, if $\mathfrak{g}$ is a simple Lie algebra on $K$.
\begin{example}
\label{EX3}
Let $ A $ be an associative algebra and let  $\text{Lie}(A)$ be the  Lie algebra with bracket $[x, y]=x\circ y-y\circ x$ for $x, y\in \text{Lie}(A)$ associated with $A$. Then,  $\text{Lie}(A)$ is a Lie $2$-algebra with $ a^{[2]}:= a^{2}.$ In particular, $\text{Lie}(\text{End}(V)):=  \mathfrak{gl}(V)$ is a  Lie $2$-algebra, where $\text{End}(V)$ is the associative algebra of $ K$-endomorphism on $V$.
\end{example}
\begin{example}

Let $ b: V \times V \rightarrow K $ be a bilinear form and consider the subset $ \mathfrak {g} (V, b) $ of $ \mathfrak {gl} (V) $ defined by $$ \mathfrak { g} (V, b): = \{f\in\mathfrak {gl} (V): b (f (u), v) + b (u, f (v)) = 0 \, \, \, \forall u, v \in V \}. $$ Then, $(\mathfrak{g}(V,b),[2])$ is  a Lie $2$-subalgebra of $(\mathfrak{g}(V),[2])$. Indeed, take $f,g\in\mathfrak{gl}(V,b)$ and $v,w\in V$. Then,
\begin{equation*}
   \begin{split}
b([f,g](v),w) &=b((fg)(v),w)-b((gf)(v),w)
                    =-b(g(v),f(w))-b(f(v),g(w))\\
                    &=-b(v,g(f(w)))+b(v,f(g(w)))=b(v,[f,g](w)).
 \end{split}
 \end{equation*}
This fact shows that $\mathfrak{g}(V, b)$ is a Lie subalgebra of $\mathfrak {gl}(V)$. Moreover,
\[b(f^{[2]}(v),w)=b(f(f(v)),w)=b(f(v),f(w))=b(v,f(f(w)))=b(v,f^{[2]}(w)).\]

 Therefore, $ f^{[2]} \in \mathfrak{g}(V, b) $, for all $ f \in \mathfrak{gl}(V). $ 
 
 \medskip
 
  It will be useful to have the matricial version of $ \mathfrak{g}(V, b) $. Given $ A \in \mathfrak{gl}_{n}(K) $, consider \[\mathfrak {g} (A) = \{X \in\mathfrak{gl}_{n}(K): X^{T}A = AX\} .\]
  
Let $ \Theta = \{v_1, v_2, ..., v_n \} $ be a basis of $ V $ and assume that $ A \in\mathfrak {gl}_{n}(K)$ is the Gram matrix of $ b $ with respect to the basis $ \Theta $, that is, $$ a_ {ij} = b (v_i, v_j) \, \, \, 1 \leq i, j \leq n. $$ So,  
 $\mathfrak{g}(V,b)$ and $\mathfrak{g}(A)$ are isomorphic as Lie $2$-algebra.
 
 \medskip
 
Two matrices $ A, B \in\mathfrak{gl}_{n}(K) $, are said to be congruent if there is $S\in \text{GL}_ {n}(K)$ such that  \[S^{T}AS = B.\] In  this case,  the map $\mathfrak{g}(A)\to\mathfrak{g}(B)$ given by $ X\mapsto S^{- 1}XS$ is a Lie $2$-isomorphism.
\end{example}

 \subsection{ Maximal Tori and Toral Rank.}
 \begin{definition} Let $(\mathfrak{g},[2])$ be a Lie $2$-algebra. An element $t\in\mathfrak{g}$ is said to be a \textit{toral element} if $t^{[2]} = t$.  A subalgebra $\mathfrak{t}$ of $(\mathfrak{g},[2])$ is called \textit{toral} (or a torus of $\mathfrak{g}$) if the $2$-mapping is invertible on $\mathfrak{t}$.\end{definition}

 Any toral subalgebra of $\mathfrak{g}$ is abelian and admits a basis consisting of toral elements (see eg. \cite{ja}).
A torus $\mathfrak{t}$ of $\mathfrak{g}$ is called \textit{maximal} if the inclusion $\mathfrak{t}\subseteq \mathfrak{t'}$ with $\mathfrak{t'}$ toral implies $\mathfrak{t'} =\mathfrak{ t}.$

\medskip

Let $(\mathfrak{g},[2])$ be a simple Lie $2$-algebra over an algebraically closed field $K$ and let $\mathfrak{h}$ be a Cartan subalgebra.
The set of toroidal elements in $\mathfrak{h}$ generates a torus. We denote this torus by the symbol $T(\mathfrak{h}).$ The torus $T(\mathfrak{h})$ is maximal in $(\mathfrak{g},[2])$ (see \cite{SP}, Lemma 4. ).
\begin{definition}(See \cite{ST2}).
The \textbf{toral rank} of a Lie $2$-algebra $(\mathfrak{g},[2])$ is \[MT (\mathfrak{g}) := \text{max}\{\text{dim}_{K}(\mathfrak{t})\,\, |\,\, \mathfrak{t}\text{ is a torus in }\; \mathfrak{g} \} \]
\end{definition}

%\section{Some simple Lie $2$-algebras.}

\section{Special Linear Lie $2$-algebra $(\mathfrak{sl}_{n+1}(K), [2])$.}
In this section we consider the Lie algebra consisting of matrices of trace zero over $K$, and we study some properties concerning about simplicity of this algebra.

\medskip

It is a known fact that the commutator of the Lie general algebra $\mathfrak{gl}_{n+1}(K)$ is a Lie subalgebra of $\mathfrak{gl}_{n+1}(K)$. This algebra is called the Lie special algebra, and it is denoted by  $\mathfrak{sl}_{n+1}(K)$, That is, 
$$\mathfrak{sl}_{n+1}(K):=[\mathfrak{gl}_{n+1}(K),\mathfrak{gl}_{n+1}(K)]=\{A\in\mathfrak{gl}_{n+1}(K):\text{tr}(A)=0\}.$$ 
It is easy to prove that 
$$\mathfrak{sl}_{n+1}(K)=\{A\in\mathfrak{gl}_{n+1}(K):\text{tr}(A)=0\}.$$ 
 A basis for $\mathfrak{sl}_{n+1}(K)$ isthe following: $$h_{k}:=e_{kk}+e_{k+1,k+1},\,\,\, k=1,...,n;\quad e_{ij},\quad i\neq j=1,2,...,n+1.$$
Let us consider the $2$-map  $[2]:\ \mathfrak{sl}_{n+1}(K)\to\mathfrak{sl}_{n+1}(K)$ given by $A^{[2]}:=A^{2}$.
\medskip

\begin{remark}
  The Lie $2$-algebra $\mathfrak{sl}_{2}(K)$ is not simple, since \[\mathfrak{sl}_{2}(K)^{(1)}=[\mathfrak{sl}_{2}(K),\mathfrak{sl}_{2}(K)]=\text{span}\{ h_1\}\] then $\mathfrak{sl}_{2}(K)^{(1)}$ is a nontrivial ideal of  $\mathfrak{sl}_{2}(K).$ In next theorem, we consider the case where $n\geq 2$.
\end{remark}

\begin{theorem}
\label{isa 1}
The special Lie algebra $\mathfrak{sl}_{n+1}(K)$ has the following properties:

\begin{enumerate}
\item  $(\mathfrak{sl}_{n+1}(K),[2])$ is a Lie $2$-algebra.
\item If $n\geq2$ and $(n+1)\not\equiv0(\text{mod}\,2)$, then  $\mathfrak{sl}_{n+1}(K)$ is a simple Lie $2$-algebra.
\item   $\mathfrak{psl}_{2n}(K):=\mathfrak{sl}_{2n}(K)/\mathfrak{z}(\mathfrak{gl}_{2n}(K))$, $n\geq1$ is a simple Lie $2$-algebra.
%\item $MT(A_n)=n$\,\, and\,\, $MT(\bar{A_{n}})=n-1.$
%\item  $dim_{K}(\mathfrak{sl}_{n}(K))=n^2-1.$
\end{enumerate}
\end{theorem}
\begin{proof}
In order to prove (1), it is enough to see that $\mathfrak{sl}_{n}(K)$ is closed by the $2$-map $[2]\colon \mathfrak{sl}_{n+1}(K)\to\mathfrak{sl}_{n+1}(K)$. But, it is an immediate consequence of the fact that   $\text{tr}(A^2)=\text{tr}(A)^2$,  for all $ A\in\mathfrak{sl}_{n+1}(K).$

\medskip

Let us prove (2).  Firstly,  we show that if  $(n+1)\not\equiv0(mod\,2)$, then $$\mathfrak{gl}_{n+1}(K)=\mathfrak{sl}_{n+1}(K)\oplus\mathfrak{z}(\mathfrak{gl}_{n+1}(K)).$$ Indeed, if $A\in\mathfrak{gl}_{n+1}(K)$ and  $\text{tr}(A)=\lambda$, then \[A=(A-\frac{\lambda}{n+1} I_{n+1})+\frac{\lambda}{n+1} I_{n+1},\] where $(A-\frac{\lambda}{n+1} I_{n+1})\in\mathfrak{sl}_{n+1}(K)$ and $\frac{\lambda}{n+1} I_{n+1}\in\mathfrak{z}(\mathfrak{gl}_{n+1}(K)).$ 
If $A\in\mathfrak{sl}_{n+1}(K)\cap\mathfrak{z}(\mathfrak{gl}_{n+1}(K))$, then $A=\lambda I_{n+1}$ and $\text{tr}(A)=(n+1)\lambda=0$. Since $n+1$  is not a multiple of $2$, we have $\lambda=0$. Therefore,  $\mathfrak{sl}_{n+1}(K)\cap\mathfrak{z}(\mathfrak{gl}_{n+1}(K))=\{0\}$. Now, let 
 $I$ be an ideal of $\mathfrak{sl}_{n+1}(K)$. Then
\begin{equation*}
\begin{split}
[\mathfrak{gl}_{n}(K),I] &=[\mathfrak{sl}_{n+1}(K)\oplus\mathfrak{z}(\mathfrak{gl}_{n+1}(K),I]\\
&=[\mathfrak{sl}_{n+1}(K),I]+[\mathfrak{z}(\mathfrak{gl}_{n+1}(K)),I]\\
                    &\subseteq I + 0=I.
 \end{split}
  \end{equation*} 
Therefore, $I$ is also  an ideal of $\mathfrak{gl}_{n}(K)$. However,  the only ideals of $\mathfrak{gl}_{n}(K)$   contained in
 $\mathfrak{sl}_{n+1}(K)$ are $\{0\}$ and $\mathfrak{sl}_{n+1}(K)$ (see \cite{ja1}). Then,  $I=\{0\}$ and $I=\mathfrak{sl}_{n+1}(K)$. Hence, $\mathfrak{sl}_{n+1}(K)$ is a simple Lie $2$-algebra. 
 
 \medskip
 
We now prove (3). If  $(n+1)\equiv0(\text{mod}\ 2)$, then $\mathfrak{z}(\mathfrak{gl}_{n+1}(K))\subseteq \mathfrak{sl}_{n+1}(K)$ is an ideal of $\mathfrak{sl}_{n+1}(K)$. Therefore,  $\mathfrak{sl}_{n+1}(K)/\mathfrak{z}(\mathfrak{gl}_{n+1}(K))$ is a Lie $2$-algebra with $2$-map given by $$(x+\mathfrak{z}(\mathfrak{gl}_{n+1}(K)))^{[2]}:=x^{[2]}+\mathfrak{z}(\mathfrak{gl}_{n+1}(K)), \quad\text{for all }x\in \mathfrak{sl}_{n+1}(K).$$ 
Now, if  $J$ is another ideal of  $\mathfrak{sl}_{n+1}(K)/\mathfrak{z}(\mathfrak{gl}_{n+1}(K))$, then $J=I/\mathfrak{z}(\mathfrak{gl}_{n+1}(K))$, where $I$ is an ideal of $\mathfrak{sl}_{n+1}(K)$ and $\mathfrak{z}(\mathfrak{gl}_{n+1}(K))\subseteq I.$ Suppose that $I\neq\mathfrak{z}(\mathfrak{gl}_{n+1}(K))$. Then, by direct computations,  we find that $e_{kl}\in I$, with $k\neq l$. Using the identities
\begin{align*}
[e_{kl},e_{lk}]&:=e_{kk}+e_{ll},\quad k\neq l\\ 
[e_{k+1,l},e_{l,k+1}]&:=e_{k+1,k+1}+e_{ll},\quad k\neq l
\end{align*}
we obtain that $h_{k}:=e_{k,k}+e_{k+1,k+1}\in I$ for all $k=1,2,...,n$. Therefore, $I=\mathfrak{sl}_{n+1}(K),$  and $\mathfrak{sl}_{n+1}(K)/\mathfrak{z}(\mathfrak{gl}_{n+1}(K))$ is a simple Lie $2$
-algebra.
\end{proof}

\medskip

%\begin{remark}
%  The Lie $2$-algebra $\mathfrak{sl}_{2}(K)$ is not simple, since \[\mathfrak{sl}_{2}(K)^{(1)}=[\mathfrak{sl}_{2}(K),\mathfrak{sl}_{2}(K)]=\text{span}\{ h_1\}\] then $\mathfrak{sl}_{2}(K)^{(1)}$ is a nontrivial ideal of  $\mathfrak{sl}_{2}(K).$
%\end{remar

Recall some well known facts about quadratic forms over an algebraically closed field of characteristic $2$ and its corresponding Lie algebras.
Let $V$ be a $n$-dimensional $K$-space and $b : V\times V \rightarrow K$ be a non-degenerate symmetric bilinear form. This means that $b(x, y) = b(y, x)$, for all $x, y \in V$
and $b(x, V ) = 0$ implies $x = 0$. A non-degenerate symmetric bilinear form $b$ is called \textit{symplectic} if $b(x, x) = 0$. Otherwise, it is called an \textit{orthogonal} bilinear form. 

\section{The Lie $2$-algebra
$(\mathfrak{g}(V,b),[2])$  with $b$ symplectic bilinear form}

In this section we study the simplicity of Lie algebra which preserves a bilinear symplectic form over $K$.\\

Let $ b: V \times V \rightarrow K $ be a symplectic bilinear form. From Example \ref{EX3},  we have that $\mathfrak{g}(V, b)$ is a Lie $2$-algebra. We denote this algebra  by $\mathfrak{sp}(V,b)$, and is called the symplectic Lie $2$-algebra. In \cite{k}, it is shown that the dimension of $V$ is even, that is,  $ n = 2m $ and there exists a basis $\beta$ of $V$ in which $b$ has Gram matrix
\[J_{2m}:= \left( \begin{array}{ccc} 
0 & I_{m}  \\
-I_{m} & 0 \\
 \end{array}\right).\]
 
 \medskip
 
 The  Lie $2$-algebra $\mathfrak{sp}(V,b)$ is isomorphic to the Lie $2$-algebra
  \[\mathfrak{sp}_ {2m}(K):=\mathfrak{g}(J_{2m})= \left\{ \left(\begin{array}{cc} 
a & b \\
c & a^{T} \end{array}\right):a,b,c\in \mathfrak{gl}_{m}(K),\, b,\,\,c\,\, \text{symmetric}\right\},\] 
which has dimension $2m^{2} +m$ and a basis consisting of the following elements:
\medskip
\begin{align*} 
d_i&:=e_{ii}+e_{m+i,m+i}, \hspace{0.8cm} (1\leq i\leq m).\\
a_{ij}&:=e_{ij}+e_{m+j,m+i}, \hspace{0.8cm} (1\leq i,j\leq m, i\neq j).\\
b_{ij}&:=e_{i,j+m}+e_{j,i+m}, \hspace{0.8cm} (1\leq i<j\leq m, i\neq j).\\
b_i&:=e_{i,i+m}, \hspace{2.2cm} (1\leq i\leq m).\\
c_{ij}&:=e_{i+m,j}+e_{j+m,i}, \hspace{0.8cm} (1\leq i<j\leq m, i\neq j).\\
c_i&:=e_{i+m,i}, \hspace{2.2cm} (1\leq i\leq m).
\end{align*}

\medskip

 The Lie bracket of $\mathfrak{sp}_ {2m}(K)$ is given in Table 1, where the elements of the diagonal are results of the $2$-map  in the elements of their rows and corresponding columns.
 
   \begin{table}[!ht]
    \begin{center}
   \begin{tabular}{|c||c|c|c|c|c|c|}
   \hline
    & $ d_i$ & $a_{ij}$ & $b_{ij}$ & $c_{ij}$ & $b_i$ & $c_{i}$\\
   \hline\hline
   $d_i$  &$d_i$ &$a_{ij}$ &$b_{ij}$ &$c_{ij}$ &$0$ &$0$\\
   \hline
   $a_{ij}$ &$a_{ij}$ &$0$ &$0$ &$0$ &$0$ &$a_{ij}$\\
   \hline
   $b_{ij}$ &$b_{ij}$ & $0$ & $0$ & $d_i+d_j$ & $0$ &$a_{ij}$\\
   \hline
   $c_{ij}$ &$c_{ij}$ &$0$ &$d_i+d_j$ &$0$ &$a_{ij}$ &$0$\\
   \hline
   $b_i$ & $0$ &$0$ &$0$ &$a_{ij}$  &$0$  &$d_i$\\
   \hline
   $c_i$ & $0$ &$c_{ij}$ &$a_{ij}$ &$0$ &$d_i$  &$0$\\
   \hline
 \end{tabular}
  \caption{The Lie $2$-algebra $\mathfrak{sp}_{2m}(K).$}
 \end{center}
 \label{TB1}
 \end{table}
 We now calculate the derived algebras of $\mathfrak{sp}_{2m}(K)$, and then, we show that the second derived algebra is a Lie $2$-algebra whenever $2$ does not divided $m$ and $m\geq 3$.
 
 \begin{remark}

 For $m=1$, we have  $\mathfrak{sp}_{2}(K)=\text{span}\{d_1,b_{12},c_{21}\}=\mathfrak{sl}_{2}(K).$ Then
 $$\mathfrak{sp}_{2}(K)^{(1)}=\text{span}\{h_1\}=\mathfrak{z}(\mathfrak{gl}_2(K))\quad\text{ and }\quad \mathfrak{sp}_{2}(K)^{(2)}=\{0\}, $$   
 and for $m=2$, we have  $\mathfrak{sp}_{4}(K)=\text{span}\{d_{1},d_{2},a_{12},a_{21},b_{12}, b_{1},b_{2}, c_{12}, c_{1}, c_2 \}$.
By direct computations we obtain that
\begin{align*}
\mathfrak{sp}_{4}(K)^{(1)}&=\text{span}\{d_{1},d_{2},a_{12},a_{21}, b_{12}, c_{12}\}.\\
\mathfrak{sp}_{4}(K)^{(2)}&=\text{span}\{d_1+d_2, a_{12}, a_{21}, b_{12},c_{12}\}.\\
\mathfrak{sp}_{4}(K)^{(3)}&=\mathfrak{z}(\mathfrak{gl}_{4}(K))\quad \text{and}\quad \mathfrak{sp}_{4}(K)^{(4)}=\{0\}.
\end{align*}

Therefore, if $m=1,2$ then $\mathfrak{sp}_{2m}(K)$ is a solvable Lie $2$-algebra.
\end{remark}
\medskip

\begin{lemma}
If $m\geq3$, then:
\begin{enumerate}
 \item  $\mathfrak{sp}_{2m}(K)^{(1)}= \left\{ \left(\begin{array}{cc} 
a & b \\
c & a^{T} \end{array}\right):b,c\in \text{Alt}_{m}(K), a\in \mathfrak{gl}_{m}(K)  \right\}$. 
 \item  $\mathfrak{sp}_{2m}(K)^{(2)}=\left\{ \left(\begin{array}{cc} 
a & b \\
c & a^{T} \end{array}\right):b,c\in \text{Alt}_{m}(K)\quad \text{and}\quad  \text{tr}(a)=0\right\}$
 \item  $\mathfrak{sp}_{2m}(K)^{(3)}=\mathfrak{sp}_{2m}(K)^{(2)},$

 \end{enumerate} 
 where  $\text{Alt}_{m}(K)$ is the set of alternating $m\times m$-matrices  with entries in $K$.
\end{lemma}

\begin{proof}
To prove (1), set $\mathfrak{g}_1:=\left\{ \left(\begin{array}{cc} 
a & b \\
c & a^{T} \end{array}\right):b,c\in \text{Alt}_{m}(K) \right\}$
and and take $\alpha$, $\beta$ in $\mathfrak{sp}_{2m}(K)$. Then there are $a$, $\bar{a}$, $b$, $\bar{b}$, $c$ and $\bar{c}$ in  $\mathfrak{gl}_{m}(K)$, where $b$, $\bar{b}$ $c$ and $\bar{c}$ are symmetric matrices, such that:
\[\alpha=\left(\begin{array}{cc} 
a & b \\
c & a^{T} \end{array}\right)\quad\quad\text{ and }\quad \quad\beta=\left(\begin{array}{cc} 
\bar{a} & \bar{b} \\
\bar{c} & \bar{a}^{T} \end{array}\right).\] Then 
%\begin{equation*}
%\begin{split}
$$[\alpha,\beta]=  \left(\begin{array}{cc} 
                                                           a\bar{a}+b\bar{c}+\bar{a}a+\bar{a}c & a\bar{b}+b\bar{a}^{T}+\bar{a}b+\bar{b}a^{T}\\
                                                           c\bar{a}+a^{T}\bar{c}+\bar{c}a+\bar{a}^{T}c
                                                           & c\bar{b}+a^{T}\bar{a}^{T}+\bar{c}b+\bar{a}^{T}a^{T}\end{array}\right).$$ Since $b$ y $\bar{b}$ are symmetric matrices, we have
  \[(a\bar{a}+b\bar{c}+\bar{a}a+\bar{b}c)^{T}=c\bar{b}+a^{T}\bar{a}^{T}+\bar{c}b+\bar{a}^{T}a^{T}\]
  \[(a\bar{b}+b\bar{a}^{T}+\bar{a}b+\bar{b}a^{T})^{T}=a\bar{b}+b\bar{a}^{T}+\bar{a}b+\bar{b}a^{T}\]  and 
  \[ (a\bar{b}+b\bar{a}^{T}+\bar{a}b+\bar{b}a^{T})_{ii}=\sum_{j=1}^{m} (a_{ij}\bar{b}_{ji}+b_{ij}\bar{a}_{ij}+\bar{a}_{ij}b_{ji}+\bar{b}_{ij}a_{ij})=0.\]                                            
Analogously, the symmetry of $c$ and $\bar{c}$ imply  
$$(c\bar{a}+a^{T}\bar{c}+\bar{c}a+\bar{a}^{T}c)^{T}= c\bar{a}+a^{T}\bar{c}+\bar{c}a+\bar{a}^{T}c,$$
$$(c\bar{b}+a^{T}\bar{a}^{T}+ \bar{c}b+\bar{a}^{T}a^{T})^{T}= c\bar{b}+a^{T}\bar{a}^{T}+\bar{c}b+\bar{a}^{T}a^{T},$$
$$(c\bar{a}+a^{T}\bar{c}+\bar{c}a+\bar{a}^{T}c)_{ii}=0.$$ Therefore, $(a\bar{b}+b\bar{a}^{T}+\bar{a}b+\bar{b}a^{T})$ and $(c\bar{a}+a^{T}\bar{c}+\bar{c}a+\bar{a}^{T}c)$ belong to $\text{Alt}_{m}(K)$. So,  $\mathfrak{sp}_{2m}(K)^{(1)}\subseteq \mathfrak{g}_1$. \\

Now, we show that  $\mathfrak{g}_1\subseteq\mathfrak{sp}_{2m}(K)^{(1)}$. Given $a=(x_{ij})\in\mathfrak{gl}_{m}(K) $ we have
$$\left(\begin{array}{cc} a & 0\\
0 & a^{T} \end{array}\right)=\\\sum_{i=1}^{m}x_{ii}\left[b_i, c_i\right]+\sum_{i\neq j}^{m}x_{i j}\left[b_i, c_{ij}\right]\in\mathfrak{sp}_{2m}(K)^{(1)}.$$\\
Let us consider the linear map $\varphi:\mathfrak{gl}_{m}(K)\rightarrow \text{Alt}_{m}(K)$ given by  $a\mapsto a+a^{T}$. Since $\text{Ker}(\varphi)=\{a\in\mathfrak{gl}_{m}(K):a\text{ is  symmetric}\}$ and $$\text{dim}_K(\text{Im}(\varphi))=m^{2}-\frac{m(m+1)}{2}=\frac{m(m-1)}{2}=\text{dim}_K(\text{Alt}_{m}(K)), $$ 
we conclude that $\text{Im}(\varphi)=\text{Alt}_{m}(K).$ That is, $\varphi$ is a surjective map. 
%Moreover, the function is a surjective linear transformation, and  . Since , 
Then, given $b\in \text{Alt}_{m}(K)$ there exists  $a\in\mathfrak{gl}_{m}(K)$ such that,  $a+a^{T}=b$. Hence,  $$\left(\begin{array}{cc}0 & b\\0 & 0\end{array}\right)=\left(\begin{array}{cc}0 & a+a^{T}\\0 & 0\end{array}\right)=\left[\left(\begin{array}{cc}a & 0\\0 & a^{T}\end{array}\right),\left(\begin{array}{cc} 0 & I \\ 0 & 0\end{array}\right)\right] \in\mathfrak{sp}_{2m}(K)^{(1)}.$$  Similarly, we prove that $\left(\begin{array}{cc}0 & 0\\c & 0\end{array}\right)\in\mathfrak{sp}_{2m}(K)^{(1)}.$  Therefore,   $\mathfrak{g}_1\subseteq\mathfrak{sp}_{2m}(K)^{(1)}$.

\medskip

To prove (2), let $:\mathfrak{g}_2= \left\{ \left(\begin{array}{cc} 
a & b \\
c & a^{T} \end{array}\right):b,c\in \text{Alt}_{m}(K)\quad \text{and}\quad\text{tr}(a)=0\right\}$. We will prove that  $\mathfrak{sp}_{2m}(K)^{(2)}= \mathfrak{g}_2$. From the description of $\mathfrak{sp}_{2m}(K)^{(1)}$ in (1), we deduce that the Lie algebra $\mathfrak{sp}_{2m}(K)^{(1)}$ is generated by $a_{ij},  b_{ij},  c_{ij}$ and $d_i$ for $1\leq i, j\leq m$. Therefore, $\mathfrak{sp}_{2m}(K)^{(2)} =  [\mathfrak{sp}_{2m}(K)^{(1)}, \ \mathfrak{sp}_{2m}(K)^{(1)} ] $ is generated by  $a_{ij}, \ b_{ij}, \ c_{ij}$ and $d_i+d_j$ for $1\leq i, j\leq m$. Since all of these elements belong to $\mathfrak{g}_2$, we conclude that $\mathfrak{sp}_{2m}(K)^{(2)}\subseteq\mathfrak{g}_2$. The another inclusion is established reasoning in a similar way to the proof of   (1).

\medskip

Finally, we prove (3).   In the proof of (2), it is proven that $\mathfrak{sp}_{2m}(K)^{(2)}$ is generated by $d_i+d_j$ for $1\leq i, j\leq m$. Therefore, 
$$\mathfrak{sp}_{2m}(K)^{(3)}=\text{span}\{ [x, y]  :  x, y\in \{a_{ij},  b_{ij},  c_{ij},  d_i+d_j : 1\leq i, j\leq m \}  \}.$$
From Table \ref{TB1}, we conclude that
$$\mathfrak{sp}_{2m}(K)^{(3)}=\text{span} \{a_{ij},   b_{ij},  c_{ij},  d_i+d_j : 1\leq i, j\leq m \}= \mathfrak{sp}_{2m}(K)^{(2)}.$$
\end{proof}

\medskip

\begin{lemma}
\label{lmS}
Let $I$ be a nontrivial ideal of $\mathfrak{sp}_{2m}(K)^{(2)}$, then $c_{ij},  b_{ij} \notin I$, for all $i,j$.
%\label{lmS}
\end{lemma}
\begin{proof}
Let $1\leq i\neq j\leq n$ fixed.  If  $c_{ij}\in I$, then for all $k\neq j$, the relations $[c_{ij},b_{ij}]=d_i+d_j, \ [d_{i}+d_j,a_{ik}]=a_{ik}, \ [d_{i}+d_j,c_{ik}]=c_{ik}$,  and $[d_{i}+d_j, b_{ik}]=b_{ik}$ imply $a_{ik}, \  b_{ik}$, \  $d_{i}+d_k$ belong to $I$ for all $1\leq k\leq m$. Since $I$ is an ideal of $\mathfrak{sp}_{2m}(K)^{(2)}$,  for all $l, k$ with $l\neq k$, we have $[a_{il},c_{ik}]=c_{lk}, [b_{il},c_{ik}]=a_{lk}$ and $[d_{i}+d_l,b_{lk}]=b_{lk}$ belong to $I$. Therefore, $I=\mathfrak{sp}_{2m}(K)^{(2)}$ which is a contradiction. Similarly, if suppose that $b_{ij}\in I$, we arrive to a contradiction. Hence, $c_{ij},  b_{ij} \notin I$ for all $i, j$.
\end{proof}

\medskip

\begin{theorem} 
\label{isa 2}
Let $b:V\times V\rightarrow K$ be a  symplectic  bilinear form and let $\mathfrak{sp}(V,b)$ be the sympletic Lie algebra associated to $b$. 
  Suppose that  $\text{dim}_k(V)=n>4$,  then:
  \begin{enumerate}
 \item  $(\mathfrak{sp}(V,b)^{(2)},2)$ is a Lie $2$-algebra.
% \item  If $4\nmid n$, $C_{\frac{n}{2}}:=\mathfrak{sp}(V,b)^{(2)}$ is simple.
 \item  If $4\mid n$, then  $\mathfrak{sp}(V,b)^{(2)}/\mathfrak{z}(\mathfrak{gl}(V))$ is simple.
 %\item  $\mathfrak{sp}(V,b)^{(2)}=\{\}$
 %\item  $dim_{K}(\mathfrak{sp}(V,b)^{(2)})=?$
 \item  If $4\nmid n$, then $\mathfrak{sp}(V,b)^{(2)}$ is simple.
 %\item  $MT(C_{\frac{n}{2}})=?$\,\, y\,\, $MT(\bar{C}_{\frac{n}{2}})=?$
 \end{enumerate}
\end{theorem}

\begin{proof}
%\begin{enumerate}
In order to prove (1), we need only to prove that $\mathfrak{sp}_{2m}(K)^{(2)}$ is closed under the $2$-map. Let $n=2m$ and $\alpha=\left(\begin{array}{cc} 
a & b \\
c & a^{T} \end{array}\right)\in \mathfrak{sp}_{2m}(K)^{(2)}$. Then $b, \ c\in \text{Alt}_m(K)$ and $\text{tr}(a)=0$. Since $b$ and $c$ are symmetric matrices, we obtain that \[\alpha^{2}=\left(\begin{array}{cc} 
a^{2}+bc & ab+ba^{T} \\
ca+a^{T}c & cb+ (a^{T})^{2} \end{array}\right)=\left(\begin{array}{cc} 
a^{2}+bc & ab+ba^{T} \\
ca+a^{T}c &(bc+ a^{2})^{T} \end{array}\right),\]
  $(ab+ba^T)^T=ab+ba^T$, and  $(ca+a^{T}c)^T=ca+a^{T}c$. Moreover, by using the equalities $b_{ii}=0, \ c_{ii}=0 $ and $\text{tr}(a)=0$, we have \[(ca+a^{T}c)_{ii}=0,\quad  (ab+ba^T)_{ii}=0,\quad \text{and}\quad  \text{tr}(a^{2}+bc)=0.\]
  Therefore, $\alpha^{2}\in\mathfrak{sp}(V,b)^{(2)}$ for all  $\alpha\in \mathfrak{sp}(V,b)^{(2)} $. Hence $\mathfrak{sp}_{2m}(K)^{(2)}$ is a Lie $2$-algebra.
  
  \medskip

Let us prove (2). If $4\mid {2m}$, then $\mathfrak{z}(\mathfrak{gl}_{2m}(K))\subseteq\mathfrak{sp}_{2m}(K)^{(2)}$ is an ideal of $\mathfrak{sp}_{2m}(K)^{(2)}$. Let  $J$ be an ideal of  $\mathfrak{sp}_{2m}(K)^{(2)}/\mathfrak{z}(\mathfrak{gl}_{2m}(K))$, then $J=I/\mathfrak{z}(\mathfrak{gl}_{2m}(K)$, where $I$ is an ideal of $\mathfrak{sp}_{2m}(K)^{(2)}$ and $\mathfrak{z}(\mathfrak{gl}_{2m}(K))\subseteq I.$  Suppose that  $I\neq\mathfrak{sp}_{2m}(K)^{(2)}.$ By Lemma \ref{lmS}, we have $c_{ij}$, $b_{ij}$ $\notin I$, therefore given $\alpha \in I $, there exists $a=(a_{ij})\in\mathfrak{g}_{m}(K)$ with $\text{tr}(a)=0$ such that 
 \[\alpha=\left(\begin{array}{cc} 
a & 0 \\
0 & a^{T} \end{array}\right).\]
Now, since $I$ is an ideal of $\mathfrak{sp}_{2m}(K)^{(2)}$, we get that \[[\alpha,X]\in I,\quad \forall X \in \mathfrak{sp}_{2m}(K)^{(2)}.\]
In particular, for \[X=\left(\begin{array}{cc} 
0 & b \\
0 & 0 \end{array}\right)\]
with $b:=e_{ij}+e_{ji}\in \text{Alt}_{m}(K)$, we have  $a_{ij}=0$ for  $i\neq j$ and $a_{ii}=a_{11}$ for all $i$. Then $\alpha=a_{ii}I_{2m}\in\mathfrak{z}(\mathfrak{gl}_{2m}(K)).$ Hence,  $I=\mathfrak{z}(\mathfrak{gl}_{2m}(K))$ and $\mathfrak{sp}_{2m}(K)^{2}/ \mathfrak{z}(\mathfrak{gl}_{2m}(K))$ is simple.
%implies that $ab+ba^{T}=0$. Therefore, taking   ,   we obtain
\medskip

Finally we prove (3). Let $I$ be an ideal of $\mathfrak{sp}_{2m}(K)^{(2)}$ , $I\neq\mathfrak{sp}_{2m}(K)^{(2)}$. Reasoning as the proof of item $(2)$, we get that $a=\lambda I_{m}$ with $m$ odd. As $\text{tr}(a)=0$, we have $\lambda=0$. Then $\alpha=0$  and $I=\{0\}$. Therefore, $\mathfrak{sp}_{2m}(K)^{(2)}$ is simple.  
%\end{enumerate}
\end{proof}

\section{Lie $2$-algebras $(\mathfrak{g}(V,b),[2])$  with $b$ orthogonal bilinear form}

In this section we show that the Lie algebra which preserve the orthogonal linear form over $K$is not a Lie $2$-subalgebra of $\mathfrak{gl}_{n}(K)$. \\

Suppose that $ b: V\times V\rightarrow K$ is an orthogonal bilinear form, and let $\mathfrak{o}(V, b)$ be the  Lie $2$-algebra associated to $b$. In \cite{k} (Theorem 20), it is shown that there exists a basis $β$ of $V$ in which  $b$ has  Gram matrix  $ D = \text{diag} (d_1 , d_2, ..., d_n) $, where $ 0\neq d_i \in K $  for all $i$,  then \[\mathfrak{g}(D):=\{A\in\mathfrak{gl}_{n}(K):d_ia_{ij}=d_ja_{ji},\quad 1\leq i,j\leq n \}.\]
Since $K$ is an algebraically closed field, we have that $K^{2}=K$, this is, every element of $K$ is a square. Then, we can assume that $D=I_n$, then \[\mathfrak{o}_{n}(K):=\mathfrak{g} (I_{n})=\{A \in \mathfrak{gl}_{n}(K):\, A \quad \text{is}\quad \text{symmetric} \}\] is a Lie $2$-algebra with basis $\{e_ {ii} \}\cup \{\bar{e}_{ij}: = e_{ij} + e_{ji} \} $, $ 1 \leq i <j \leq n $ and
 whose Lie bracket is given by:
\begin{align*}
[\bar{e}_{ii},\bar{e}_{ij}]&=\bar{e}_{ij},\,\,\,\, 1\leq i<j\leq n.\\ [\bar{e}_{ii},\bar{e}_{kk}]&=0,\,\,\, i\neq k.\\
[\bar{e}_{ij},\bar{e}_{kl}] &=\delta_{ik}\bar{e}_{jl}+\delta_{il}\bar{e}_{kj}+\delta_{jk}\bar{e}_{il}+\delta_{jl}\bar{e}_{ik}\,\,\,\, \forall i<j, k<l.
\end{align*}
  Moreover, $\bar{e}_{ij}^{2}=e_{ii}+e_{jj}$,  and $\text{dim}_{K}(\mathfrak{o}_n(K))=\frac{n(n+1)}{2}.$

\medskip

\begin{lemma}
$\mathfrak{o}_{n}(K)^{(1)}=\emph{Alt}_{n}(K)$ and  $\emph{dim}_{K}(\mathfrak{o}_{n}(K)^{(1)})=\frac{n(n-1)}{2}.$
\label{lmA}
\end{lemma}
\begin{proof}
Let $a,b\in\mathfrak{o}_{n}(K)$, then $$[a,b]^{T}=(ab-ba)^{T}=b^{T}a^{T}-a^{T}b^{T}=-(ab-ba)=[a,b].$$
Thus,  $[a,b]$ is a symmetric matrix. Moreover, by the symmetry of $a$ and $b$, we have $[a,b]_{ii}=\sum_j(a_{ij}b_{ji}-b_{ij}a_{ji})=0$. Therefore,  $\mathfrak{o}_{n}(K)^{(1)}\subseteq \text{Alt}_n(K).$ Reciprocally the matrices $\bar{e}_{ij}:=e_{ij}+e_{ji}$, where $1\leq i<j\leq n$, form a basis of  $\text{Alt}_n(K)$. Now, since $\bar{e}_{ij}=[\bar{e}_{ik},\bar{e}_{kj}]\in \mathfrak{o}_{n}(K)^{(1)}$ for all $i, j$, we have  $\mathfrak{o}_{n}(K)^{(1)}=\text{Alt}_{n}(K)$ and $\text{dim}_{K}(\mathfrak{o}_{n}(K)^{(1)})=\text{dim}_K(\mathfrak{o}_{n}(K))-n=\frac{n(n-1)}{2}.$
\end{proof}
\begin{remark}
From Lemma \ref{lmA}, it follows that the following elements 
$$\bar{e}_{ij}=e_{ij}+e_{ji} $$
for $i\neq j$ form a basis of $\mathfrak{o}_{n}(K)^{(1)}$. Now, since $\bar{e}_{ij}^2=e_{ii}+e_{jj}$ does not belong to $\mathfrak{o}_{n}(K)^{(1)}$, we have $\mathfrak{o}_{n}(K)^{(1)}$ is not a Lie $2$-algebra with respect to the $2$-map $\left[ 2 \right]: \mathfrak{o}_{n}(K)\to \mathfrak{o}_{n}(K)$  defined by $a\mapsto a^2$.
\end{remark}

\section{Classical type simple Lie $2$-algebra and their toral rank.}
 W. Killing and E. Cartan show that all simple Lie algebra over an algebraically closed field of zero characteristic is isomorphic to one of the Classical algebras of Lie $A_{n}$ $(n\geq 1)$, $B_{n}$ $(n\geq2)$, $C_{n} (n\geq3)$, $D_{n}(n\geq 4)$ or to the Exceptional Lie algebras, $\mathfrak{g}_{2},\mathfrak{f}_{4},\mathfrak{e}_{6},\mathfrak{e}_{7},\mathfrak{e}_{8}$ (see \cite{ja}), but in characteristic 2, it seems that many new phenomena arise, for instante, these are not necessarily simple,  or some of them are isomorphic and, and therefore,   the classification of simple Lie algebras over the field $K$ will be different from those of the characteristics 0 and $p\geq 5$. In this section, we calculate the toral rank of the simple $2$-Lie algebra of the classical type and we conclude that there are no classical type simple Lie $2$-algebra  of odd toral rank. In particular, there are no classical type simple Lie $2$-algebra  of  toral rank 3.

\begin{definition}
Given an irreducible root system of type $X_{l}$ and its corresponding Chevalley algebra $\mathfrak{g}(X_{l},{K})$ over the field $K$, the quotient \[\overline{\mathfrak{g}(X_{l},{K})}:=\mathfrak{g}(X_{l},{K})/\mathfrak{z}(\mathfrak{g}(X_{l},K)),\] where $\mathfrak{z}(\mathfrak{g}(X_{l},K))$ is the center of $\mathfrak{g}(X_{l},K)$, is usually called the \textit{classical Lie algebra of type $X_{l}$}.
\end{definition}

\begin{remark}
This definition is exactly the same as Steinberg's, but Steinberg excluded some types of characteristic $2$ and $3$.
\end{remark}

The simplicity of the classical type Lie algebras in
characteristic $2$ have been determined  by Hogeweij in $[6]$, as indicated in the following theorem.

\begin{proposition}
\label{isa 3}
Suppose that $X_{l}$ is a Lie algebra which is not of type $A_{1}$, $B_{l}$, $C_{l}$, or $\mathfrak{f}_{4}$. Then $\overline{\mathfrak{g}(X_{l},{K})}$ is a simple Lie $2$-algebra. 
\end{proposition}
%\begin{proposition}

%\end{proposition}
So, from Proposition \ref{isa 3}, Theorem \ref{isa 1}  and  Theorem \ref{isa 2} it follows that the classical type simple Lie $2$-algebra  are:
\begin{corollary} 
\label{COR 2,4}
The  classic type simple Lie $2$-algebra are:
\begin{enumerate}
    \item Type $A_{n}$:
    \begin{align*} 
 \overline{\mathfrak{g}(A_{n},K)}&\simeq \mathfrak{sl}_{n+1}(K)\quad \text{if}\quad 2\nmid(n+1),\\ 
\overline{\mathfrak{g}(A_{n},K)}&\simeq \mathfrak{psl}_{2n}(K)\quad \text{if}\quad 2\mid(n+1).
\end{align*}
    
%\begin{itemize}
  %  \item  $\overline{\mathfrak{g}(A_{l},\mathbb{F})}\simeq \mathfrak{sl}_{l+1}(\mathbb{F})$ if $2\nmid(l+1).$
  % \item $\overline{\mathfrak{g}(A_{l},\mathbb{F})}\simeq \mathfrak{psl}_{2l}(\mathbb{F})/\mathfrak{z}(\mathfrak{gl}_{2l}(\mathbb{F}))$ if $2\mid(l+1).$
%\end{itemize}
\item Type $D_{n}$:
\begin{align*} 
\overline{\mathfrak{g}(D_n,K)}&\simeq \mathfrak{sp}_{2n}(K)^{(2)}\quad \text{if}\,\, n\,\, \text{is}\,\, \text{odd}, \\
\overline{\mathfrak{g}(D_n,K)}&\simeq \mathfrak{sp}_{2n}(K)^{(2)}/\mathfrak{z}(\mathfrak{gl}_{n}(K))\quad \text{if}\,\, n\,\, \text{is}\,\, \text{even}.
\end{align*}
%\begin{itemize}
%\item $\overline{\mathfrak{g}(D_l,\mathbb{F})}\simeq \mathfrak{sp}_{2l}(\mathbb{F})$ if $l$ is odd .
%\item $\overline{\mathfrak{g}(D_l,\mathbb{F})}\simeq \mathfrak{sp}_{2l}(\mathbb{F})/\mathfrak{z}(\mathfrak{gl}_{l}(F))$ if $l$ is even.
%\end{itemize}
\item Type $\mathfrak{g}_{2}$:
\begin{align*} 
\overline{\mathfrak{g}(\mathfrak{g}_{2},K)}&=\mathfrak{g}(\mathfrak{g}_{2},K).  
\end{align*}

%\begin{itemize}
% \item $\overline{\mathfrak{g}(\mathfrak{g}_{2},\mathbb{F})}=\mathfrak{g}(\mathfrak{g}_{2},\mathbb{F}).$
%\end{itemize}
\item Type $\mathfrak{e}_{6}$:
\begin{align*} 
\overline{\mathfrak{g}(\mathfrak{e}_{6},K)}=\mathfrak{g}(\mathfrak{e}_{6},K).
\end{align*}

%\begin{itemize}
%\item $\overline{\mathfrak{g}(\mathfrak{e}_{6},\mathbb{F})}=\mathfrak{g}(\mathfrak{e}_{6},\mathbb{F}).$
%\end{itemize}

\item Type $\mathfrak{e}_{7}.$
\begin{align*} 
\overline{\mathfrak{g}(\mathfrak{e}_{7},K)}%&=\overline{\mathfrak{g}(\mathfrak{e}_{7},\mathbb{F})}. 
\end{align*}

%\begin{itemize}
%\item $\overline{\mathfrak{g}(\mathfrak{e}_{7},\mathbb{F})}.$
%\end{itemize}
\item Type $\mathfrak{e}_{8}.$
%\begin{itemize}
\begin{align*}
\overline{\mathfrak{g}(\mathfrak{e}_{8},K)}&=\mathfrak{g}(\mathfrak{e}_{8},K).
\end{align*}
 %\end{itemize}
 \end{enumerate}
 \end{corollary}

\begin{theorem}
\label{PROPO23}
Let $\mathfrak{g}$ be a classical type simple Lie $2$-algebra and $\mathfrak{h}$ be a Cartan subalgebra  of $\mathfrak{g}$. Then $$MT(\mathfrak{g})=\emph{dim}_{K}(\mathfrak{h}).$$
\end{theorem}
\begin{proof} Let $\mathfrak{g}$ be a classical type simple Lie $2$-algebra. Then from Corollary \ref{COR 2,4} it follows that   $\mathfrak{g}=\overline{\mathfrak{g}(X_{l},K)}$ with $X_{l}\neq A_{1},B_{l},C_{l},\mathfrak{f}_{4}$. Hence, any quotient of the form
\[\overline{\mathfrak{h}(X_{l},K)}=\mathfrak{h}(X_{l},K)/\mathfrak{z}(\mathfrak{g}(X_{l},K)),\] where $\mathfrak{h}(X_{l},K)$ is a Cartan subalgebra of the Chevalley $K$-algebra  $\mathfrak{g}(X_{l},K)$ is a Cartan subalgebra of $\mathfrak{g}$.
Since $\mathfrak{h}(X_{l},K)=\text{span}\{h_{i}\otimes 1:h_{i}\in\mathfrak{h}_{X_{l}}\}$ and  $\mathfrak{h}_{X_l}$ is the subalgebra of diagonal matrices of $\mathfrak{sl}_{l+1}(K)$, we obtain 
$(h_{i}\otimes 1)^{[2]}=h_{i}\otimes 1$, for each 
$h_{i}\in \mathfrak{h}_{X_{l}}$. Thus, the equality $([h_{i}\otimes 1])^{[2]}=[h_{i}\otimes 1]$ module $\mathfrak{z}(\mathfrak{g}(X_{l},K))$ implies that
$\overline{\mathfrak{h}(X_{l},K)}\subseteq T(\overline{\mathfrak{h}(X_{l},K)})$,
and as $T(\overline{\mathfrak{h}(X_{l},K)})\subseteq\overline{\mathfrak{h}(X_{l},K)}$, we have that 
$\overline{\mathfrak{h}(X_{l},K)}= T(\overline{\mathfrak{h}(X_{l},K)})$. 
Since any pair of Cartan  Lie subalgebra of  a finite-dimensional classical type Lie algebra   $\mathfrak{g}$ over $K$ are conjugate (see \cite{p}),  there exists an automorphism $\sigma \in \text{Aut}(\mathfrak{g})$ such that $\mathfrak{h}=\sigma(\overline{\mathfrak{h}(X_{l},K)})$. Then,  from Lemma 5 ( see \cite{SP}), we obtain 
\[T(\mathfrak{h})=\sigma(T(\overline{\mathfrak{h}(X_{l},K)}))=\sigma(\overline{\mathfrak{h}(X_{l},K)})=\mathfrak{h}.\]
Then any Cartan subalgebra of a simple Lie $2$-algebra $\mathfrak{g}$ of classical type is a maximal tori in $\mathfrak{g}$, hence \[MT(\mathfrak{g})=\text{dim}_{K}(\mathfrak{h}).\]
\end{proof}

\medskip

\vspace{0.0cm}
A direct consequence of Theorem \ref{PROPO23} is the following. 
\begin{corollary}
\label{COR}
The toral rank of the  classical type simple Lie $2$-algebras is: 
\begin{enumerate}
\item $MT(A_{n})=n$,\,\,\,\, \hspace{0.8cm} if\,\, $2\nmid (n+1)$ \,\,$l>1.$
\item $MT(A_{n})=n-1$, \hspace{0.4cm} if\,\,\, $2\mid(n+1)$, \,\,\,$n>1.$
\item $MT(D_n)=n-1,$\hspace{0.5cm} if\,\, $n$\,\, is\,\, odd,\,\,\,$n\geq3.$
\item $MT(D_n)=n-2,$ \hspace{0.4cm} if\,\,\, $n$\,\, is\,\, even,\,\,\,$n\geq3.$
\item $MT(\mathfrak{g}_{2})=2.$
\item $MT(\mathfrak{e}_{6})=6.$
\item $MT(\mathfrak{e}_{7})=6.$
\item $MT(\mathfrak{e}_{8})=8.$
\end{enumerate}
\end{corollary}

From  Corollary \ref{COR}, it follows the following result.

\medskip

\begin{theorem}
There are no classical type simple Lie $2$-algebra  of odd toral rank.
\end{theorem}

\section{A (contragredient) simple Lie $2$-algebra of dimension $34$ and toral rank $4$.}

In this section we show that the contragredient Lie $2$-algebra $G(F_{4, a})$ constructed by V. Kac and V. Ve\u{i}sfe\u{i}ler (see \cite{ka}) has toral rank $4$, and we obtain the Cartan decomposition of this algebra.

\begin{definition}
Given an $(n\times n)$-matrix $A=(a_{ij})$ with elements in $K$, we denote by $\widetilde{G}(A)$ the Lie algebra determined by the system of generators $e_{i}$, $f_{i}$, $h_{i}$, $i=1,...,n$, and the system of relations 
\begin{equation}
    \label{R}
[e_{i},f_{j}]=\delta_{ij}h_{j},\quad [h_i,h_j]=0, \quad [h_i,e_j]=a_{ij}e_{j}, \quad [h_{i},f_{j}]=-a_{ij}f_{j},
\end{equation}
for $1\leq i,j\leq n.$ 
We set $\text{deg }   e_{i}=1$, $\text{deg }  f_{i}=-1$ and $\text{deg }  h_{i}=0$, $i=1,2,\dots,n$. Thus, the algebra $\widetilde{G}(A)$ becomes into a graded Lie algebra,  $\widetilde{G}(A)=\oplus_{i\in\mathbb{Z}}\widetilde{G}_{i}.$ Let $J(A)$ be a maximal homogeneous ideal in $\widetilde{G}(A)$ such that $J(A)\cap(\widetilde{G}_{-1}\oplus \widetilde{G}_{0}\oplus \widetilde{G}_{1})=0$. The Lie algebra $G(A):=\widetilde{G}(A)/J(A)$ is called a \textit{contragredient Lie algebra} and  $A$ is its Cartan matrix. 
\end{definition}

In \cite{ka},  V. Kac\, and\, V. Ve\u{i}sfe\u{i}ler considered the algebra $$G(F_{4,a}):=\widetilde{G}(F_{4,a})/J(F_{4,a}),$$  where \[F_{4,a}:= \left( \begin{array}{cccc}
0 & 1 & 0 & 0\\
a & 0 & 1 & 0 \\
0 & 1 & 0 & 1 \\
0 & 0 & 1 & 0 \end{array} \right),\] 
 $a\in K\setminus\{0,1\}$ and $J(F_{4,a})$ is the only maximal homogeneous ideal in $\widetilde{G}(F_{4,a})$ such that
 \[
\begin{cases}
J(F_{4,a})\cap\text{span}\{h_1,h_2,h_3,h_4\}  =0\\
J(F_{4,a}) \cap\text{span} \{e_i,f_j\}  =0.
\end{cases}
\]

They  proved that $G(F_{4,a})$ is a  simple Lie $2$-algebra of dimension $34$ with Cartan matrix $F_{4,a}$, with $a\in K\setminus\{0,1\}$ (see \cite{ka}, Proposition $3.6$). We now prove that this $2$-algebra Lie has toral rank $4$ and, furthermore, we give its Cartan decomposition.

%\end{example}

% $p_1:=e_1e_2$, $p_2:=e_1e_3$, $p_3:=e_1e_4$, $p_4:=e_2e_3$, $p_5:=e_2e_4$, $p_6:=e_3e_4$, $p_1e_3:=(e_1e_2)e_3$, $p_1e_4:=(e_1e_2)e_4$, $p_2e_4:=(e_1e_3)e_4$, $p_4e_4:=(e_2e_3)e_4$, $p:=(e_1e_2)e_3)e_4$, $q_1:=f_1f_2$, $q_2:=f_1f_3$, $q_3:=f_1f_4$, $q_4:=f_2f_3$, $q_5:=f_2f_4$, $q_6:=f_3f_4$, $q_1f_3:=(f_1f_2)f_3$, $q_1f_4:=(f_1f_2)f_4$, $q_2f_4:=(f_1f_3)f_4$, $p_4f_4:=(f_2f_3)f_4$, e $q:=(f_1f_2)f_3)f_4$.\\
%A base for $\widetilde{G}(F_{4,a})$ is:
%\begin{equation*}
 %  \begin{split}
 %\Phi &:=\{h_i\}^4_{i=1}\cup\{e_i\}^4_{i=1}\cup\{p_i\}^6_{i=1}\cup\{p_1e_3, p_1e_4, p_2e_4, p_4e_4, p\}\cup\\
%&\quad\cup\{f_i\}^4_{i=1}\cup\{q_i\}^6_{i=1}\cup\{q_1e_3, q_1e_4, q_2e_4, q_4e_4, q\}.  
% \end{split}
% \end{equation*} 
 %Therefore, a base of $G(F_{4,a})$ is $\{\bar{w}:=w+J(F_{4,a}):w\in\Phi\}$ and
 
 \medskip
 
From \eqref{R}, we conclude that $$\mathfrak{h}:=\text{span}\{\bar{h}_{i}:=h_i+J(F_{4,a}):i\in I_{4}\}$$ is a Cartan subalgebra of $G(F_{a,4})$.
 We now explicitly describe the maximal tori $T(\mathfrak{h})$ consisting of toroidal elements in $ \mathfrak{h}$. 
 
Since $h_1^{[2]}\in \mathfrak{h}$, we have $h_1^{[2]}=\delta_1h_1+\delta_2h_2+\delta_3h_3+\delta_4h_4$, and by using the relations \eqref{R} we obtain $$0=[h_1,[h_1,e_1]]=\delta_1[h_1,e_1]+\delta_2[h_2,e_1]+\delta_3[h_3,e_1]+\delta_4[h_4,e_1]=\delta_2ae_1,$$  which implies that $\delta_2=0$. Similarly, we obtain $\delta_2=\delta_3=\delta_4=0$ and $\delta_1=1$. So, $h_1^{[2]}=h_1$
\\
%$e_2=[h_1,[h_1,e_2]]=\delta_1[h_1,e_2]+\delta_2[h_2,e_2]+\delta_3[h_3,e_2]+\delta_4[h_4,e_2]=\delta_1e_2+\delta_3e_2$, then $\delta_1+\delta_3=1$.\\
%$0=[h_1,[h_1,e_3]]=\delta_2e_3+\delta_4e_3$, then $\delta_2=\delta_4=0$.\\
%$0=[h_1,[h_1,e_4]]=\delta_3e_4$, then $\delta_3=0$. So, $h_1^{[2]}=h_1$.\\ 
We also find  $h_3^{[2]}=h_3, \  h_{4}^{[2]}=h_{4}$ and  $h_2^{[2]}=ah_2+\bar{a}h_4$, where $\bar{a}=a+1$ 

%$h_2^{[2]}=\delta_1h_1+\delta_2h_2+\delta_3h_3+\delta_4h_4$, then
%$a^{2}e_1=[h_2,[h_2,e_1]]=\delta_1[h_1,e_1]+\delta_2[h_2,e_1]+\delta_3[h_3,e_1]+\delta_4[h_4,e_1]$, then $a^{2}e_1=\delta_2ae_1$ therefore, $\delta_2=a$.
%$0=[h_2,[h_2,e_2]]=[h_2^{[2]},e_2]=\delta_1e_2+\delta_3e_2$, then $\delta_1+\delta_3=0,$
%$e_3=[h_2,[h_2,e_3]]=\delta_2e_3+\delta_4e_3$, then $\delta_4+\delta_2=1$. So, $\delta_4=1+a:=\bar{a}$.
%$0=[h_2,[h_2,e_4]]=[h_2^{[2]},e_4]=\delta_3e_4$, then $\delta_3=0$
%therefore, $h_2^{[2]}=ah_2+\bar{a}h_4.$ For the same reason we have
%$h_3^{[2]}=h_3$, $h_{4}^{[2]}=h_{4}$ . 
  Let $t_2:=xh_2+yh_4$, with $x,y\in K$. If the equality  $t_2^{[2]}=t_2$ holds true, then $x, y$ satisfy the following system of equations
  \[  \left\{
\begin{array}{ll}
ax^2+x & =0
 \\
\bar{a}x^2+y^2+y & =0, 
 \end{array}
\right.
\]
whose solution set is 
$$\{(0,0), (0,1),(\frac{1}{a},\frac{1}{a}), (\frac{1}{a},\frac{\bar{a}}{a}) \}.$$
First two solutions give $t_2=0$, and $t_2=h_4$ respectively, and with  the last two solutions we obtain $t_2=\frac{1}{a}(h_2+h_4)$ and  $t_2=\frac{1}{a}(h_2+\bar{a}h_4)$. Since $\frac{1}{a}(h_2+h_4)=\frac{1}{a}(h_{2}+\bar{a}h_{4})-h_4$, we have  $$T(\mathfrak{h}):=\text{span}\{\bar{h}_1,\bar{h}_3,\bar{h}_4,\frac{1}{a}(h_2+h_4)+J(F_{4,a})\}$$ and $\text{dim}_{K}(T(\mathfrak{h}))=4.$ 
This fact shows that $MT(G(F_{4,a}))=4$.

\medskip

 We now find the Cartan decomposition of $G(F_{4,a})$ with respect to $T(\mathfrak{h})$. By definition of the ideal $J(F_{4, a})$, the elements $e_i, \ f_i$ and $h_i$ for $1\leq i\leq 4$ does not belong to $J(F_{4, a})$. Therefore, the classes $\overline{e}_i=e_i+J(F_{4, a}),  \overline{f}_i=f_i+J(F_{4, a})$ and $\overline{e}_i=h_i+J(F_{4, a})$ for $1\leq i\leq 4$, belong to a basis for $G(F_{4,a})$. Now, to complete a basis for $G(F_{4,a})$, we consider the product $xy:=\left[x, y \right]$. The  products $xy$, where $x$ and $y$ are generators of $G(F_{4,a})$, and some of them does not belong to $\{e_i, f_i : 1\leq i \leq 4\}$ are zero or belong to $\text{span}\{h_i, e_i, f_i : 1\leq i \leq 4\}$. Thus, the only products of two generators that give us new generators are $e_ie_j$ and $f_if_j$ with $i<j$. So, the elements $\overline{p_{ij}}=e_ie_j+J(F_{4, a})$ and $\overline{q_{ij}}=f_if_j+J(F_{4, a})$ with $i<j$ are also generators of $G(F_{4, a}),$ which are linearly independent with $e_ie_j+J(F_{4, a})$ and $f_if_j+J(F_{4, a})$. Reasoning in a similar way we obtain that the elements $(e_1e_2)e_3, (e_1e_2)e_4, (e_1e_3)e_4, (e_2e_3)e_4, ((e_1e_2)e_3)e_4$ modulo $J(F_{4, a})$ and $(f_1f_2)f_3, (f_1f_2)f_4, (f_1f_3)f_4, (f_2f_3)f_4, ((f_1f_2)f_3f)_4$ modulo $J(F_{4, a})$ complete a basis for $G(F_{4, a})$. We denote this basis by $\Phi$.\\
 
Next, we calculate the weights for each element of the basis $\Phi$ of $\widetilde{G}(F_{4,a})$ with respect to $$\mathfrak{h}_{1}:=\{h_1,\frac{1}{a}(h_2+h_4),h_3,h_4\}$$ are:

\begin{itemize}
  \item 
$[e_1,h_1]=a_{11}e_1=0e_1$,\\
$[e_1,\frac{1}{a}(h_2+h_4)]=\frac{1}{a}([e_1,h_2]+[e_1,h_4]=\frac{1}{a}(ae_1)=1e_1$,\\
$[e_1,h_3]=a_{31}e_1=0e_1$,\\
$[e_1,h_4]=a_{41}e_1=0e_1$.\\
Then, the weight of $\overline{e}_1$ is $\beta:=(0,1,0,0).$

\medskip
\item 
$[e_2,h_1]=a_{12}e_2=1e_2$,\\
$[e_2,\frac{1}{a}(h_2+h_4)]=\frac{1}{a}([e_2,h_2]+[e_2,h_4]=0e_2$,\\
$[e_2,h_3]=a_{32}e_2=1e_2$,\\
$[e_1,h_4]=a_{41}e_1=0e_1$.\\
The weight the $\overline{e}_2$ is $\alpha+\gamma:=(1,0,1,0).$

\medskip
\item 
$[e_3,h_1]=a_{13}e_3=0e_3$,\\
$[e_3,\frac{1}{a}(h_2+h_4)]=\frac{1}{a}([e_3,h_2]+[e_3,h_4]=0e_3$,\\
$[e_3,h_3]=a_{33}e_3=0e_3$,\\
$[e_3,h_4]=a_{43}e_3=1e_3$.\\
 The weight the $\overline{e}_3$ is $\lambda:=(0,0,0,1).$
 
 \medskip
\item 
$[e_4,h_1]=a_{14}e_4=0e_4$,\\
$[e_4,\frac{1}{a}(h_2+h_4)]=\frac{1}{a}([e_4,h_2]+[e_4,h_4])=0e_4$,\\
$[e_4,h_3]=a_{34}e_4=1e_4$,\\
$[e_4,h_4]=a_{44}e_4=0e_4$.\\
Then, the  weight of $\overline{e}_4$ is $\gamma:=(0,0,1,0).$
 \end{itemize}
 
By the  similarity in the definition of the bracket $[h_i, f_j]=a_{ij}f_j$ with the bracket $[h_i, e_j]=a_{ij}e_j$, we deduce that $\overline{e}_i$ and $\overline{f}_i$, for $1\leq i\leq 4$, have the same weight. On the other hand, by using $[\mathfrak{g}_{\xi},\mathfrak{g}_{\mu}]\subseteq \mathfrak{g}_{\xi+\mu}$, we obtain that the remaining elements of $\Phi$ have the weights given in Table \ref{tabla:sencilla}, where we use the notation $\overline{p_{ijk}}=(e_ie_j)e_k+ J(F_{4, a})$,  $\overline{q_{ijk}}=(f_if_j)f_k+ J(F_{4, a})$,  $\overline{p_{1234}}=((e_1e_2)e_3)e_4+ J(F_{4, a})$ and $\overline{q_{1234}}=((f_1f_2)f_3)f_4+ J(F_{4, a})$.

\begin{table}[!ht]
\begin{center}
%\scalefont{0.9}
\begin{tabular}{|l|l|l|}
\hline
Root spaces & Basis & Roots \\
\hline \hline
$\mathfrak{g}_\beta$ & $\overline{e}_1$,\:  $\overline{f}_1$ & $\beta:=(0,1,0,0)$ \\ \hline
$\mathfrak{g}_{\alpha+\gamma}$ & $\overline{}e_2$, \: $\overline{f}_2$ & $\alpha+\gamma:=(1,0,1,0)$ \\ \hline
$\mathfrak{g}_\lambda $ & $\overline{e}_3$, \: $\overline{f}_3$ & $\lambda:=(0,0,0,1)$ \\ \hline
$\mathfrak{g}_\gamma $ & $\overline{e}_4$, \: $\overline{f}_4$ & $\gamma:=(0,0,1,0)$ \\ \hline
$\mathfrak{g}_{\alpha+\beta+\gamma}$ & $\overline{p_{12}}$,\:  $\overline{q_{12}}$ & $\alpha+\beta+\gamma:=(1,1,1,0)$ \\ \hline
$\mathfrak{g}_{\beta+\lambda}$ & $\overline{p_{13}}$,\:  $\overline{q_{13}}$ & $\beta+\lambda:=(0,1,0,1)$ \\ \hline
$\mathfrak{g}_{\beta+\gamma}$ & $\overline{p_{14}}$,\:  $\overline{q_{14}}$ & $\beta+\gamma:=(0,1,1,0)$ \\ \hline
$\mathfrak{g}_{\alpha+\gamma+\lambda}$ & $\overline{p_{23}}$,\:  $\overline{q_{23}}$ & $\alpha+\gamma+\lambda:=(1,0,1,1)$\\ \hline
$\mathfrak{g}_\alpha$ & $\overline{p_{24}}$,\:  $\overline{q_{24}}$ & $\alpha:=(1,0,0,0)$ \\ \hline
$\mathfrak{g}_{\gamma+\lambda}$ & $\overline{p_{34}}$,\:  $\overline{q_{34}}$ & $\gamma+\lambda:=(0,0,1,1)$ \\ \hline
$\mathfrak{g}_{\alpha+\beta+\gamma+\lambda}$ & $\overline{p_{123}}$,\:  $\overline{q_{123}}$  & $\alpha+\beta+\gamma+\lambda:=(1,1,1,1)$ \\ \hline
$\mathfrak{g}_{\alpha+\beta}$ &  $\overline{p_{124}}$,\:  $\overline{q_{124}}$  & $\alpha+\beta:=(1,1,0,0)$ \\ \hline
$\mathfrak{g}_{\beta+\gamma+\lambda}$ & $\overline{p_{134}}$,\:  $\overline{q_{134}}$  & $\beta+\gamma+\lambda:=(0,1,1,1)$ \\ \hline
$\mathfrak{g}_{\alpha+\lambda}$ &  $\overline{p_{234}}$,\:  $\overline{q_{234}}$  & $\alpha+\lambda:=(1,0,0,1)$ \\ \hline
$\mathfrak{g}_{\alpha+\beta+\lambda}$ & $\overline{p_{1234}}$, \: $\overline{q_{1234}}$ & $\alpha+\beta+\lambda:=(1,1,0,1)$ \\ \hline
\end{tabular}
\caption{Root spaces of $G(F_{4,a})$.}
\label{tabla:sencilla}
\end{center}
\end{table}
%\vspace{0.2cm}
 Therefore, the Cartan decomposition of $G(F_{4,a})$ with respect to $T(\mathfrak{h})$ is  \[G(F_{4,a})=T(\mathfrak{h})\oplus\left(\underset{\xi\in G}\oplus\mathfrak{g}_\xi\right),\] where $G:=\left\langle \alpha,\beta,\gamma,\lambda\right\rangle$ is an elementary abelian group of order $16$, and $\text{dim}_k(\mathfrak{g}_\xi)=2$, for all $\xi\in G$.
 Therefore, we have:
 
 \medskip
 
\begin{theorem}
The  contragredient  Lie algebra $G(F_{4,a})$ on $K$
with Cartan matrix  \[F_{4,a}:= \left( \begin{array}{cccc} 
0 & 1 & 0 & 0\\
a & 0 & 1 & 0 \\
0 & 1 & 0 & 1 \\
0 & 0 & 1 & 0 \end{array} \right)\] 
have the following properties:
\begin{enumerate}
\item $G(F_{4,a})$ is a simple Lie $2$-algebra of dimension $34.$
\item $MT(G(F_{4,a}))=4.$
\item  the Cartan descomposition of $G(F_{4,a})$ with respect to $T(\mathfrak{h})$ is  \[G(F_{4,a})=T(\mathfrak{h})\oplus\left(\underset{\xi\in G}\oplus\mathfrak{g}_\xi\right),\] where $G:=\left\langle \alpha,\beta,\gamma,\lambda\right\rangle$ is an elementary abelian group of order $16$, and $\text{dim}_k(\mathfrak{g}_\xi)=2$, for all $\xi\in G$.
\end{enumerate}
\end{theorem}


\begin{thebibliography}{99}

%\bibitem {A} T. Aoki, \textit{Calcul exponentiel des op\'erateurs
%microdifferentiels d'ordre infini.} I, Ann. Inst. Fourier (Grenoble)
%\textbf{33} (1983), 227--250.

%\bibitem {B} R. Brown, \textit{On a conjecture of Dirichlet},
%Amer. Math. Soc., Providence, RI, 1993.

%\bibitem {D} R. A. DeVore, \textit{Approximation of functions},
%Proc. Sympos. Appl. Math., vol. 36,
%Amer. Math. Soc., Providence, RI, 1986, pp. 34--56.

%\bibitem[Do]{gz} DOLOTKAZIN, A., \textit{Irreducible representations of a three-dimensional simple Lie algebras of characteristic $p=2$}, Math. Note $24$  $(1978)$, $588-590$. $6$
%\bibitem[Gr]{gr} GRICHKOV, A.N., \textit{On simple Lie algebras over a field of characteristic 2,} J. Algebra $363$ $(2012)$,$14-18.$
\bibitem{gr} GRISHKOV, Alexander; PREMET, Alexander.  Simple Lie algebras of absolute toral rank 2 in characteristic 2, \textit{preprint}.
%\bibitem{gr} Grishkov, A.N., Premet, A.A., \textit{Simple Lie algebras of absolute toral rank 2 in characteristic 2}. Preprint.
\bibitem{gr21345} GRISHKOV, Alexander. On simple Lie algebras over a field of characteristic 2. \textit{Journal of Algebra}, 2012, vol. 363, p. 14-18.
%\bibitem[GGA]{gg} GRICHKOV, A.N., ARAUJO, W., \textit{As \'algebras de Lie simples de dimens\~ao $7$ sobre um corpo de caracter\'istica $2$ e suas sub\'algebras toroidais}, IME-USP, 2014.
%\bibitem[GZ]{gz} GRICHKOV, A.N., ZUSMANOVICH,P.,\textit{Deformations of current Lie algebras. I. Small algebra in characteristic $2$}, arXiv:1410.3645
\bibitem{ja} JACOBSON, Nathan. Lie algebras, Interscience, New York, 1962. 1979.
%\bibitem{JA} Jacobson, N., \textit{Lie Algebra}, Interscience Publishers, New York ($1962$).
\bibitem{ja1} JACOBSON, Nathan. Abstract derivation and Lie algebras. \textit{Transactions of the American Mathematical
Society}, vol. $42$ $(1937)$, pp. $216-217$.
%\bibitem[Ja1]{ja1} JACOBON, N., \textit{Basic Algebra II.} Dover Book on Mathematics, second edition, $2009$. 
%\bibitem[KV]{ka} KAC, V.,VE\u{I}SFE\u{I}LER, B., \textit{Exponentials in Lie algebras of characteristic $p$, Math.USSR Izvestija, vol.$5$ $(1971)$, No. $4$} 
\bibitem{gr2} SKRYABIN, Serge. Toral rank one simple Lie algebras of low characteristics. \textit{Journal of Algebra}, 1998, vol. 200, no 2, p. 650-700.

%\bibitem{gr2} Skryabin, S.,\textit{ Toral rank one simple Lie algebras of low characteristic}, J. Algebra \textbf{200} (1998), 650-700.
%\bibitem[4]{gr1} STRADE, H., FARNSTEINER, R., \textit{Modular Lie Algebra and Their Representations}, Marcel Dekker, New York, 1988.
%\bibitem[St]{kl}STRADE, H.,\textit{Lie algebras of small dimension}, htt//arXiv.org/abc/math/0601413. $17$ jan $2006$.

\bibitem{ST2} STRADE, Helmut. The absolute toral rank of a Lie algebra. \textit{Lie algebras}, Madison 1987. Springer, Berlin, Heidelberg, 1989. p. 1-28.
%\bibitem[5 ]{ta} Strade. H., \textit{Simple Lie Algebras over Fiels of Positive characteristic, Volume I: Structure Theory,} DeGruyter Expositions in Math., Vol. 38, Berlin, 2004.
\bibitem{H} G.M.D.Hogeweij. Almost-classical Lie algebra, I, \textit{Nederl. Akad. Wetensch.Indag. Math.} $44$\,\, ($1982$)\quad$441-452;$ \\
G.M.D.Hogeweij.,Almost-classical Lie algebra, II,\textit{ Nederl. Akad. Wetensch.Indag. Math.} $44$\,\, ($1982$)\quad $453-460.$
\bibitem{SP} S. P. Demuskin., Cartan subalgebras of simple Lie $p$-algebras \textit{Translated from Sibirskii Matematicheskii Zhurnal}, Vol. $11$, No. $2$, pp. $310–325$, March-April, $1970$.
\bibitem{ka} KAC, V.G, VE\u{I}SFE\u{I}LER, B., Exponentials in Lie algebras of characteristic $p$,\textit{Math.USSR Izvestija}, vol.$5$ $(1971)$, No. $4.$
\bibitem {VK} KAC, V. G., On the classification of the simple Lie algebras over a field with nonzero
characteristic, \textit{Izv.Akad. Nauk SSSR Ser. Mat. 34} (1970), $385-408$ [Russian];
Math. USSR-Izv. 4 (1970), $391-413$ [English translation].
%\bibitem {VK} KAC, V.G., On the classification of simple Lie algebras over a field of nonzero
%characteristic, \textit{Izv. Akad. Nauk SSSR Ser. Mat. 34} (1970), 385-408 = Math. USSR
%Izv. 4 (1970), 391-413. MR 43 \#2033.
%\bibitem[St2]{gr2} STRADE,H.,\textit{Simple Lie Algebras over Fiels of Positive characteristic, Volume II: Classifying the Absolute Toral Rank Two case.} DeGruyter Expositions in math., vol. 42, Berlin, 2009.  
%\bibitem[St3]{na} STRADE. H.,\textit{Simple Lie Algebras over Fiels of Positive characteristic, Volume III: Completion of the Classification}, DeGruyter Expositions in Math., Vol. 57, Berlin, 2012.
\bibitem{k}  KAPLANSKY, I.L Linear algebra and geometry. A second course,. \textit{Allyn and Bacon}, Boston, $1969$.
\bibitem{p}  SELIGMAN, G.B. On Lie algebras of prime characteristic, \textit{Memoirs of the Americ. Math. Soc.}, No.19.
\end{thebibliography}
\end{document}